\documentclass[12pt,usenames,dvipsnames]{article}
\usepackage[a4paper]{anysize}\marginsize{2cm}{2cm}{2cm}{2cm}
\pdfpagewidth=\paperwidth \pdfpageheight=\paperheight
\usepackage{amsfonts,amssymb,amsthm,amsmath,eucal,tabu,hyperref}
\usepackage{pgf}
\usepackage{array}
\usepackage{pstricks}
\usepackage{tikz-cd}
\usepackage{pstricks-add}
\usepackage{pgf,tikz,pgfplots}
\usepackage{float}
\pgfplotsset{compat=1.10}
\usetikzlibrary{automata}
\usetikzlibrary{arrows}
\usetikzlibrary{matrix}
\usepackage{indentfirst}
\usepackage{xargs} 
\usepackage{xcolor}

\newcommand{\bigslant}[2]{{\left. \raisebox{.3em}{$#1$}\middle/\raisebox{-.3em}{$#2$}\right.}}

\theoremstyle{plain}
\newtheorem{thm}{Theorem}[section]
\newtheorem{theorem}[thm]{Theorem}
\newtheorem{lemma}[thm]{Lemma}
\newtheorem{proposition}[thm]{Proposition}
\newtheorem{construction}[thm]{Construction}

\theoremstyle{definition}
\newtheorem{definition}[thm]{Definition}
\newtheorem{remark}[thm]{Remark}
\newtheorem{example}[thm]{Example}

\newtheorem{thevarthm}[thm]{\varthmname}

\newenvironment{varthm*}[1]{\trivlist\item[]{\bf #1.}\it}{\endtrivlist}

\def\keywordname{{\bfseries Keywords}}%
\def\keywords#1{\par\addvspace\medskipamount{\rightskip=0pt plus1cm
		\def\and{\ifhmode\unskip\nobreak\fi\ $\cdot$
		}\noindent\keywordname\enspace\ignorespaces#1\par}}
\def\subclassname{{\bfseries Mathematics Subject Classification
		(2010)}\enspace}
\def\subclass#1{\par\addvspace\medskipamount{\rightskip=0pt plus1cm
		\def\and{\ifhmode\unskip\nobreak\fi\ $\cdot$
		}\noindent\subclassname\ignorespaces#1\par}}

\def\P{\mathbb{P}}
\def\A{\mathcal{A}}

\def\C{\mathbb{C}}
\def\F{\mathbb{F}}
\def\al{\alpha}
\def\be{\beta}
\def\ga{\gamma}
\def\O{\mathcal{O}}
\DeclareMathOperator{\Der}{Der}
\DeclareMathOperator{\Hom}{Hom}
\DeclareMathOperator{\D}{D}
\DeclareMathOperator{\syz}{syz}
\DeclareMathOperator{\sym}{Sym}

\title{On unexpected curves of type $(d+k,d)$.
}
\author{Grzegorz Malara, Halszka Tutaj-Gasi\'nska}
\date{\today}
\begin{document}
	\maketitle
	\begin{abstract}
		We present a construction  explaining the existence of (unexpected) curves of degree $d+k$, passing through a set $Z$ of points on $\P^2$, and having a generic point $P$ of multiplicity $d$.  The construction is based on  the syzygies of the $k$-th powers of 
		Jacobian of the product of lines dual to the points of $Z$. We prove also a result characterizing the unexpectedness of the curves via splitting type of the bundle of these syzygies retricted to the line dual to
		$P$,  providing a kind of generalization of the theory started by Faenzi and Valles and by Cook II, Harbourne, Migliore and Nagel.
		
		\subclass{14N20, 14C20}
		
	\end{abstract}
\maketitle
	
\section{Introduction}
	Studying the dimension of a given linear system of divisors is one of the most classical problems in algebraic geometry. Typical examples of interesting linear systems arise when imposing vanishing conditions on divisors in  linear systems. Determining the dimension of such a system amounts to deciding if the imposed vanishing conditions are independent or not. If the underlying variety is a projective space and vanishing (to order one) is imposed in {\it general  points} then the  resulting system is either empty, or its dimension is determined by the number of points. Imposing vanishing to order two in general points of a projective space is also well understood due to a highly non-trivial result of Alexander and Hirschowitz, \cite{AH}. For points of higher multiplicity even a conjectural picture is, in general, missing. However, in the case of the complex projective plane the so-called SHGH conjecture \cite{CM},  provides a conjectural picture. Despite intensive investigations over last 40 years, the SHGH conjecture remains widely open. 
	Basing on the results of Faenzi and Valles \cite{FV, V} in 2016 Cook II, Harbourne, Migliore and Nagel in their paper \cite{CHMN}, started a new direction of research.
	They considered curves of degree $d+1$ passing through 
	a set $Z \subset \P^2$ of non-general points and  having multiplicity $d$ in a general point $P$. It is crucial here that the point $P$ is general. If the existence of such curves does not follow from the na\"ive dimension count, then the authors of \cite{CHMN} call them unexpected (of type $(d+1,d)$). Their work was motivated by findings for $d=3$ in \cite{DIV} by DiGennaro, Ilardi and Vall\`es.  Unexpected curves, and more generally, unexpected hypersurfaces, attracted a lot of attention, see eg. \cite{BMSS}, \cite{DMO}, \cite{Dimca}, \cite{DFHMST}, \cite{FGHM}, \cite{HMNT}, \cite{HMT-G}. 
	In particular, the authors of \cite{KS} provide 
	an infinite family of unexpected curves of type $(2m+1,3)$.  
	
	The present work was motivated by an attempt  to explain the existence of this family.  Then we wanted to create a more general theory, along the lines of \cite{CHMN}, explaining the existence of unexpected curves with the difference between the degree and the multiplicity in the general point greater than 1.  In \cite{CHMN} the existence of unexpected cures of type $(d+1, d)$ is explained via degree $d$ syzygies of the Jacobian ideal of an arrangement $\A_Z$ of lines dual to the points of the set $Z$. The purpose of our work is to investigate unexpected curves of type $(d+k, d)$ from this point of view. More specifically, in Theorem \ref{thm:main} we describe a construction of curves of type $(d+k,d)$ based on syzygies of the $k$-th power of a suitable twist of the Jacobian ideal of the arrangement $\A_Z$. This construction generalizes the one of Cook II, Harbourne, Migliore, Nagel from \cite{CHMN}, where they dealt with $k=1$ case. Proposition \ref{pro:unexp} provides, in turn, a numerical criterion for the curves resulting from Theorem \ref{thm:main} to be unexpected. Section \ref{sec:derivations} and Lemma \ref{lem:generaliz 3.3} are the technical heart of the work.  In Section \ref{sec:derivations} we show that  the syzygy bundle of $J^k/f$ may be treated as a  subbundle of the $k$-th symmetric powers of the tangent bundle, just generalizing, in a sense, the notion of the logarithmic derivations.
In \ref{lem:generaliz 3.3} we prove the connection between the splitting type of the above bundle and the dimension of $I_Z+jP$ in degree $j+k.$
	
	In the last section of our paper  we present examples of  (various types of) unexpected curves,
	some of them may be a starting point for new problems.

\section{Basic facts and the syzygies construction }\label{sec:construction}
	
	In this section we establish the notation and recall basic facts and definitions needed in the sequel.
	
	Let $S=\C[x,y,z]$  
	be the graded ring of polynomials over
	field $\C$. Let $\check{\P}^2$  be the  space dual to $\P^2$ and denote by $Z$ a collection of simple points in $\check{\P}^2$. We associate with the set $Z\subset \check{\P}^2$ the line arrangement $\A_Z=\A=\{H_1,\ldots,H_{|Z|}\}$, where $H_i's$ are lines dual to the points of $Z$.  We   denote by $\ell_i$ the linear forms defining $H_i$. In this paper we deal with line arrangements, therefore in the sequel $f$ always stands for $f=\ell_1\cdot \ldots \cdot \ell_{|Z|}$, which we call the defining polynomial of $\A$. 
	For a given point $Q=(a,b,c)$ in $\P^2$, we use $L_Q$ to denote the line in $\check{\P}^2$ dual to this point. Similarly, if $L \subset \P^2$ is a generic line with the equation $\al x+\be y+\ga z=0$, then $P_{L}=(\al, \be, \ga)$ indicates the dual point of $L$.
	
	Let $\Der(S)$ be the module of $\C$-derivations of $S$ and denote by $\theta_E=x\frac{\partial}{\partial x}+y\frac{\partial}{\partial y}+z\frac{\partial}{\partial z} \in \Der(S)$ the Euler derivation. For a given homogeneous element $f$ we define 
	$$\D(f) = \{\theta \in \Der(S) \; | \; \theta(f) \in fS \}.$$
	It is known that $\D(f)=S\theta_E \oplus \D_0(f), $ where $\D_0(f)$ is the kernel of the map $\partial \mapsto \partial(f)$ and is called the derivation module, giving rise to the derivation bundle, see \cite{CHMN}. We use also the notation $\D(\A)$ if $f$ is the defining polynomial of an arrangement $\A$.
	
	In the same spirit, we define  $\Der^k(S)$ to be the module 
	$$\Der^k(S)=\Big\{\theta \; | \; \theta=h_{k,0,0}\Big(\frac{\partial}{\partial x}\Big)^k+h_{k-1,1,0}\Big(\frac{\partial}{\partial x}\Big)^{k-1}\frac{\partial}{\partial y}+\ldots+h_{0,0,k}\Big(\frac{\partial}{\partial z}\Big)^k \Big\} \cong S^{\binom{k+2}{2}},$$
	where $h_{i_1,i_2,i_3}\in S$. Note, that here  $\Big(\frac{\partial}{\partial x}\Big)^{i_1}\Big(\frac{\partial}{\partial y}\Big)^{i_2}\Big(\frac{\partial}{\partial z}\Big)^{i_3}$ is a product of partial derivatives of first order, not a derivative of higher order.  
	
	In what follows $J=( f_x,f_y,f_z)$ always denotes the jacobian ideal of $f$.
	
	Let us now recall   the definition of unexpected curves.
	Let  $I_Z$ be the saturated homogeneous ideal of a finite set $Z$ of pairwise different points  in $\P^2$ and let $P$ be a generic point in $\P^2$.
	Given a homogeneous ideal $I\subseteq S$ we denote by $[I]_t$ the vector space 
	spanned by all forms in $I$ of degree $t$. Let $k$ be a positive integer.
	
	\begin{definition}
	
		We say that a curve $C$ given as a zero set of a form 
		in $[I_Z]_{d+k}$, having a point of multiplicity $d$ at $P$, is unexpected  of type $(d+k,d)$ if 
		
		$$\dim [I_{P\cup Z}]_{d+k} > \max\left(0, \dim [I_Z]_{d+k} - \binom{d+1}{2}\right);$$
		i.e., $C$ is unexpected if vanishing $d$ times in $P$ imposes on $[I_Z]_{d+k}$ fewer than the expected
		number of conditions. Moreover, we assume (as some authors do and some do not) that vanishing in $Z$ imposes independent conditions on the forms of degree $d+k$. 
		
	\end{definition}
	
	This definition may be further generalized to the case of unexpected hypersurfaces, vanishing  (unexpectedly) along some linear subspaces, see e.g. \cite{HMT-G}.
	
	As we mentioned above, the existence of  unexpected curves of type $(d+1,d)$ is explained in \cite{CHMN} by treating the curve as an image of a certain map from a line in $\P^2$, where the map is defined  with the help of  the syzygies of the jacobian ideal of $f$.

	This construction is the starting point for our paper. We describe it from a slightly changed point of view.
	
	\begin{construction}\label{prop:chmn5.10}		{\rm \cite{CHMN}
			\begin{itemize}
				\item Take a set $Z$ of pairwise different points   in a projective plane. Treat this plane as a dual projective plane $\check{\P}^2$.
				
				\item The points of $Z$ give dual lines in $\P^2$
				with the equations given by the forms $l_1,\dots,l_{|Z|}$. Let $f=\ell_1\cdots \ell_{|Z|}$. 

				\item Take a generic line $L \in \P^2$, with its dual point $P_{L}=(\al, \be, \ga)$. Take a point $Q=(a,b,c)$ in $L$.
				
				\item   Let $J$ denote the jacobian ideal of $f$ with fixed generators. Take a syzygy of $J+(L)$ (of minimal degree), say $(s_0,s_1,s_2,s_3)$. Write $s(Q)=(s_0(Q),s_1(Q),s_2(Q))$, for $Q\in L$, where $s_0,s_1,s_2$ are of degree $d$.
				
				\item  Then take two lines in the dual plane: $\ell_Q$ (dual to $Q$) and $\ell_{s(Q)}$, dual to $s(Q)$. The lines (in general) intersect in a point $P$, so we have a map: $L \ni Q\to P\in \check{\P}^2$.
				
				\item The map is not defined if $\ell_{s(Q)}=\ell_Q$. It may happen only when $Q$ is a point of intersection of $L$ and $f$.
				
				\item In \cite{CHMN} it is proved that when a point $Q$ moves along the generic line $L$, the image $P$, moves along a curve $C$ of degree $d +1$.
				
				\item The curve $C$ constructed  this way passes through all such points of $Z$ that the map is defined there, and has a point of multiplicity $d$ in $P_{L}=\check{L}=(\alpha,\beta,\gamma)$.
		
		\end{itemize}}
	\end{construction}

\section{Syzygies-based construction }\label{sec:construction-generalized}
	
	The following theorem and its proof describe a construction, generalizing the one  from Construction \ref{prop:chmn5.10}. By means of this generalized construction we will get curves of type $(d+k,d)$, for $k\geq 1$. 
	
	\begin{theorem}
		\label{thm:main}
		Let $Z$ be a set of $|Z|$ points in $\check{\P}^2$ and let $L$ be a generic line on $\mathbb{P}^2$ with the equation $\al a+\be b+\ga c=0$. Denote by $(g_{k,0,0},\dots,g_{0,0,k}, g)$ a (reduced) syzygy of $J^k+(L)$ where $g_{i_1,i_2,i_3}$ are all of degree $d$, i.e., for any $Q=(a:b:c)\in \P^2$ we have
		$$g_{k,0,0}(Q)f_x(Q)^k+g_{k-1,1,0}(Q)f_x(Q)^{k-1}f_y(Q)+\dots+g_{0,0,k}(Q)f_z(Q)^k+g(Q)L(Q)=0.$$
		Let $S_Q$ be the curve of degree $k$ in $\check{\P}^2$ given by the equation
		$$S_Q(x,y,z):=g_{k,0,0}(Q)x^k+g_{k-1,1,0}(Q)x^{k-1}y+\dots+g_{0,0,k}(Q)z^k=0.$$

		Let $Q=(a,b,c)\in L$.
		Consider in $\check{\P}^2$ the system of equations
		\begin{equation}\label{s}\tag{$\star$} 
			\begin{cases} 
				\al a+\be b+\ga c=0\\
				ax+by+cz=0\\
				g_{k,0,0}(a,b,c)x^k+g_{k-1,1,0}(a,b,c)x^{k-1}y+\dots+g_{0,0,k}(a,b,c)z^k=0.
			\end{cases}
		\end{equation}
		We will say that this system is \emph{not determined} in $Q=(a,b,c)\in L$ if for all $(x,y,z)$ we have
		$$ g_{k,0,0}(a,b,c)x^k+g_{k-1,1,0}(a,b,c)x^{k-1}y+\dots+g_{0,0,k}(a,b,c)z^k = (ax+by+cz)^k.$$
		
		Let $P_L=\check{L}=(\al,\be,\ga)$.
		
		Then:
		\begin{enumerate}
			\item The  system $(\star)$ is not determined,  only for points $Q$ on $\A \cap L$.
			
			\item The solutions $(x,y,z)$ to the system $(\star)$
			lie on a curve $C_L\subset \check{P}^2$ of degree (at most) $d+k$.

			\item $C_L$ passes through   $Z$. 

			\item $C_L$ has a point of multiplicity at least $d$ in  $P_L$.
			
			\item The curve $C_L$ may be treated as  $C_L(x,y,z)$ with parameters $(\al,\be,\ga)$ and "dually" as  $C_L(\al,\be,\ga)$ with parameters $(x,y,z)$. The partial derivatives computed in point $(\al,\be,\ga)$ with respect to $(x,y,z)$ and computed in point $(x,y,z)$ with respect to $(\al, \be, \ga)$ are the same up to order $d$.
		\end{enumerate}
	\end{theorem}

	\begin{proof}
		\begin{itemize}
			\item[]
			\item [Ad 1)]
			If the system $(\star)$ is not determined in a point $Q=(a,b,c)$  on the line $L$, then in particular, for $(x,y,z)=(f_x(Q),f_y(Q),f_z(Q))$ we have
			\begin{align*}
				0=g_{k,0,0}(Q)f_x(Q)^k+g_{k-1,1,0}(Q)f_x(Q)^{k-1}f_y(Q)+\dots+g_{0,0,k}(Q)f_z(Q)^k=\\
				=(af_x(Q)+bf_y(Q)+cf_z(Q))^k.
			\end{align*}
			From Euler's equality $af_x(Q)+bf_y(Q)+cf_z(Q)=\deg(f)f(Q)$, so $f(Q)=0.$

			\item [Ad 2)] Here we show that the curve $C_L$ passes through all such points of $Z$  where $(\star)$  is determined. Indeed, $C_L$ passes through $z_j\in Z$ if 
			$z_j\in L_Q$ and $z_j\in S_Q$ for a point $Q\in L$.
			
			Let $z_j=(z_{j0},z_{j1},z_{j2})$. 
			
			Observe that $z_j\in L_Q $  if $Q\in \ell_j$, so $Q$ must be the intersection point of $\ell_j$ and $L$. 
			For such an intersection point $Q$ we have, from  the fact that $g_{i_1,i_2,i_3}$, are syzygies of $J^k$:
			
			$$g_{k,0,0}(Q)f_x(Q)^k+g_{k-1,1,0}(Q)f_x(Q)^{k-1}f_y(Q)+\dots+g_{0,0,k}(Q)f_z(Q)^k=0$$
			
			As $f=\ell_1\cdots \ell_j\cdots\ell_{|Z|}$, we have
			$$(f_x(Q),f_y(Q),f_z(Q))=\nabla_Q f=h(Q)\nabla_Q\ell_j=h(Q)(z_{j0},z_{j1},z_{j2}),$$
			with a suitable polynomial $h$ satisfying $h(Q)\neq 0$.
			
			Thus 
			$$g_{k,0,0}(Q)z_{j0}^k+g_{k-1,1,0}(Q)z_{j0}^{k-1}z_{j1}+\dots+g_{0,0,k}(Q)z_{j2}^k=0.$$
			
			This gives that $z_j\in S_Q$, so $C_L$ passes through $z_j$.

			\item [Ad 3)] If $(\star)$ is not determined at $Q_i=L\cap l_{z_i}$, then the solutions of this system consist of  the whole line $L_{Q_i}$. Of course, $z_i\in L_{Q_i}$ (as $Q_i\in l_{z_i}$).
			
			From $2)$ and $3)$ we see that the solutions of $(\star)$ pass through all $z\in Z$. 
			
			\item [Ad 4)] The curve $C_L$ has a point of multiplicity $d$ in $P_L=\check{L}=(\al,\be,\ga)$. Indeed,  observe that as the point $P_L$ belongs   to
			every $L_Q,$ where $Q\in L$. We need only to check that $P_L$ belongs to $d$ of curves $S_Q$.
			
			Let  $H_L$ be a hyperplane in $\P^r, r=\binom{k+2}{2}-1,$ defined as
			$$H_L=\{(x_0:\dots:x_r) : \al^kx_0+\al^{k-1}\be x_1+\dots+\ga^kx_r=0\}.$$ 
			
			Then $\{G(Q):=(g_{k,0,0}(Q),\dots,g_{0,0,k}(Q)), Q\in L\}$ is a curve of degree  $d$ in $\P^r$. It cuts the hyperplane $H_L$ in (at least) $d$ points, say $Q_1,\dots, Q_d$. From the genericity of $L$, the points are different.

			For any such point $Q_i$, with $i=1,\dots d$, we have
			
			$$g_{k,0,0}(Q_i)\al^k+ g_{k-1,1,0}(Q_i)\al^{k-1}\be+\dots+g_{0,0,k}(Q_i)\ga^k=0,$$
			so $P_L\in S_{Q_i}$. 

			\item [Ad 5)] Here we describe a very explicit construction of the curve $C_L$. This part of the proof proves $5$ and is  an alternative proof of $2$ and $4$. 

			For $(a,b,c)\in L$, where $L=\al a+\be b+\ga c=0$ is a generic line on $\P^2$, the solutions of the two equations:
			
			$$ax+by+cz=0$$
			and
			$$g_{k,0,0}(a,b,c)x^k+g_{k-1,1,0}(a,b,c)x^{k-1}y+\dots+g_{0,0,k}(a,b,c)z^k=0$$
			lie on a curve of degree $d+k$.

			Indeed, take a point $(a,b,c)\in L$. As the syzygy is reduced, we may  assume without loss of generality, that $a^d$ appears in at least one $g_{i_1,i_2,i_3}$. We may also assume  that $c=1$, and that $\al , \be\neq 0,$  so
			$$b=\frac{-\al a-\ga }{\be}.$$
			Thus
			$$g_{k,0,0}(a,b,c)x^k+g_{k-1,1,0}(a,b,c)x^{k-1}y+\dots+g_{0,0,k}(a,b,c)z^k=0$$
			is equivalent to
			$$g_{k,0,0}(a,\frac{-\al a-\ga }{\be},1)x^k+g_{k-1,1,0}(a,\frac{-\al a-\ga }{\be},1)x^{k-1}y+\dots+g_{0,0,k}(a,\frac{-\al a-\ga }{\be},1)z^k=0.$$
			As $g_i$ are of degree $d$ we may reorder  the last equation to get:
			\begin{equation}\label{ag}
				a^dh_1(x,y,z)+a^{d-1}h_2(x,y,z)+\dots+h_d(x,y,z)=0
			\end{equation}
			with some homogeneous polynomials $h_j(x,y,z)$ of degree $k$ (and in general depending also on $\al, \be, \ga$).
			
			On the other hand, putting $b=\frac{-\al a-\ga }{\be}$ and $c=1$ into the first, linear equality, we get
			$$\be ax+(-\al a-\ga )y+\be z=0.$$
			Thus for all points $(x,y,z)$ except $(\al, \be, \ga)$ we have
			$$a=\frac{\ga y-\be z}{\be x-\al y}.$$

			Put this $a$ into the equation (\ref{ag}) and multiply by $(\be x-\al y)^d$.
			We get:
			
			$$ (\ga y-\be z)^dh_1(x,y,z)+(\ga y-\be z)^{d-1}(\be x-\al y)h_2(x,y,z)+\dots+h_d(x,y,z)(\be x-\al y)^d=0.$$
			
			What we obtain is a curve of degree  $d+k$,
			with a $d$-fold point at $(\al, \be,\ga)=\check{L}$.
			
			Moreover, 
			the partial derivatives computed with respect to
			$(x,y,z)$ in  the point $(x,y,z):=\check{L}$ and computed with respect to $(\al, \be, \ga)$ in  the point $(\al, \be, \ga):=\check{L}$  are the same up to order $d$.
	
		\end{itemize}
	\end{proof}
	\begin{remark}{\rm 
			
			We also compute explicitly, what we know already from the proof of the assertion $3)$ in Theorem \ref{thm:main}, that if the system $(\star)$ is not determined in a point $Q_1$ on $L$, then the line $L_{Q_1}$ is a component of $C_L$.
			Indeed, in the above construction, we take the curve
			$$g_{k,0,0}(a,b,c)x^k+g_{k-1,1,0}(a,b,c)x^{k-1}y+\dots+g_{0,0,k}(a,b,c)z^k=0,$$
			replace $a,b$ by $\frac{\ga y-\be z}{\be x-\al y}$ and $\frac{\al z-\ga x}{\be x-\al y}$ and multiply by $(\be x-\al y)^d$. We get

			\begin{equation*}\label{sl}
				g_{k,0,0}(\frac{\ga y-\be z}{\be x-\al y},\frac{\al z-\ga x}{\be x-\al y},1)x^k+g_{k-1,1,0}(\frac{\ga y-\be z}{\be x-\al y},\frac{\al z-\ga x}{\be x-\al y},1)x^{k-1}y+\dots\\  \end{equation*}
			\begin{equation}\tag{$*$}
				\dots +g_{0,0,k}(\frac{\ga y-\be z}{\be x-\al y},\frac{\al z-\ga x}{\be x-\al y},1)z^k=0. 
			\end{equation}
			
			Assume that a point, say $(a_1,b_1,1)=Q_1=L\cap l_{z_1}$, is such that the system $(\star)$ is not determined. Then
			\begin{equation}\label{eq:cond on g}
				\begin{cases}
					g_{k,0,0}(a_1,b_1,1)=a_1^k\\
					g_{k-1,1,0}(a_1,b_1,1)=ka_1^{k-1}b_1\\
					\dots\\
					g_{0,0,k}(a_1,b_1,1)=1.
				\end{cases}
			\end{equation}
			
			Put $z=-a_1x-b_1y$ into ($*$). 
			
			Then you get 
			$$\frac{\ga y-\be (-a_1x-b_1y)}{\be x-\al y}=a_1,$$
			
			$$\frac{\al (-a_1x-b_1y)-\ga x}{\be x-\al y}=b_1.$$

			So in ($*$) we get 
			$$
			g_{k,0,0}(a_1,b_1,1)x^k+g_{k-1,1,0}(a_1,b_1,1)x^{k-1}y+
			\dots +g_{0,0,k}(a_1,b_1,1)(-a_1x-b_1y)^k.$$
			Using (\ref{eq:cond on g})
			we see that this is identically zero and $L_{Q_1}$ divides $C_L$. 
		}
	\end{remark}
	
\section{Derivations}
	\label{sec:derivations}
	In this section we show a connection between the syzygies of $J^k$ and a bundle of derivations $\D_0^k(\A)$, analogous to the connection between $\syz(J/f)$ and $\D_0(\A)$. Showing this we will in Section \ref{sec:unexpecdness} establish a relation between the degree of a curve which can be obtained from the construction described in Theorem \ref{thm:main} and the exponents in the splitting type of $\D_0^k(\A)$ restricted to a general line $L$. A reader not familiar with the relation between $\syz(J/f)$ and $\D_0(\A)$ may want to see  Appendix to \cite{CHMN} for a detailed introduction to the subject (see also \cite{FV}, \cite{OT}).

	Let $J=( f_x,f_y,f_z)$. We have an exact sequence
	$$0 \longrightarrow \D^k(\A) \longrightarrow S^{\binom{k+2}{2}} \xrightarrow{\;\; \phi \;\;} (J/f)^k(|Z|-1) \longrightarrow 0,$$
	where, for an element $(g_{k,0,0},g_{k-1,1,0},\ldots,g_{0,0,k}) \in S^{\binom{k+2}{2}}$ there is
	$$\phi \left( 
	\begin{bmatrix}
		g_{k,0,0}\\
		\vdots \\
		g_{0,0,k}
	\end{bmatrix}
	\right) = g_{k,0,0}\cdot\Big(\frac{\partial f}{\partial x}\Big)^k+g_{k-1,1,0}\cdot\Big(\frac{\partial f}{\partial x}\Big)^{k-1}\frac{\partial f}{\partial y}+\ldots+g_{0,0,k}\cdot\Big(\frac{\partial f}{\partial z}\Big)^k \pmod f,$$
	and $\D^k(\A) \subset \Der^k(S) $ is the set of such derivations $\delta$ that $\delta(f) \in J^{k-1} (f)$.
	Let us remind, that here  $\Big(\frac{\partial}{\partial x}\Big)^{i_1}\Big(\frac{\partial}{\partial y}\Big)^{i_2}\Big(\frac{\partial}{\partial z}\Big)^{i_3}$ is a product of partial derivatives of first order, not a derivative of a higher order.  
	
	In order to define the main object of this section, the module $\D_0^k(\A)$, we need the following.
	\begin{definition}
		Let $i_1+i_2+i_3=k-1$. For all $j \in \big\{1,2,\ldots,\binom{k+1}{2}\big\}$ we define the derivations module $E_j\in D^k(\A)$:
		$$E_j= ( dx)^{i_1}(dy)^{i_2}(dz)^{i_3} \Big(M(k)\Big) \star P(k),$$
		where $M(k)$ and $P(k)$ are $\underbrace{k+1 \times \cdots \times k+1}_{\text{$k$}}$ dimensional matrices which consist of all monomials and derivatives of degree and rank $k$ respectively. The symbol $\star$ denotes the Hadamard product of matrices. 
	\end{definition}
	
	\begin{example}\label{ex:euler}
	
		For $k=2$ we have that 
		$$E_1 = dx (M(2)) \star P(2)= dx \left( \begin{bmatrix}
			x^2 & xy & xz \\
			yx & y^2 & yz \\
			zx & zy & z^2
		\end{bmatrix} \right) \star 
		\begin{bmatrix}
			\partial_x \cdot \partial_x & \partial_x \cdot \partial_y & \partial_x \cdot \partial_z \\
			\partial_y \cdot \partial_x & \partial_y \cdot \partial_y & \partial_y  \cdot \partial_z \\
			\partial_z \cdot \partial_x & \partial_z \cdot \partial_y & \partial_z \cdot  \partial_z
		\end{bmatrix} = $$
		$$= \begin{bmatrix}
			2x & y & z \\
			y & 0 & 0\\
			z & 0 & 0
		\end{bmatrix} \star 
		\begin{bmatrix}
			\partial_x \cdot \partial_x & \partial_x \cdot \partial_y & \partial_x \cdot \partial_z \\
			\partial_y \cdot \partial_x & \partial_y \cdot \partial_y & \partial_y \cdot \partial_z \\
			\partial_z \cdot \partial_x & \partial_z \cdot \partial_y & \partial_z \cdot \partial_z
		\end{bmatrix} = \begin{bmatrix}
			2x\partial_x \cdot \partial_x & y\partial_x \cdot \partial_y & z\partial_x \cdot \partial_z \\
			y\partial_y \cdot \partial_x & 0 & 0 \\
			z\partial_z \cdot \partial_x & 0 & 0
		\end{bmatrix},$$ which, due to symmetry of the elements of matrices, can be viewed as $E_1=[x\partial_x \cdot \partial_x, y\partial_x \cdot \partial_y, z\partial_x \cdot \partial_z, 0, 0, 0 ]. $
		
		Similarly we obtain $$E_2 =dy (M(2)) \star P(2)=[0, x\partial_x \cdot \partial_y,0, y\partial_y \cdot \partial_y, z\partial_y \cdot \partial_z, 0 ], $$ 
		$$E_3 =dz (M(2)) \star P(2)= [0,0, x\partial_x \cdot \partial_z, 0, y\partial_y \cdot \partial_z, z\partial_z \cdot \partial_z ].$$

\bigskip

    Observe that the action of $\phi$ on $E_j$ gives an element from
    $J\cdot (f)$; indeed $\phi(E_1)=f_x\cdot f$ and so on. 
    
    More generally, when $E_j\in D^k(\A)$, then $\phi$ on $\phi(E_j)=(f_x)^{i_1}(f_y)^{i_2}(f_z)^{i_3}f\in J^{k-1}\cdot f$.

	\end{example}

	\begin{definition}
		We define the module
		$$\D_0^k(\A)= \bigslant{\D^k(\A)}{SE_1 \oplus \cdots \oplus SE_{\binom{k+1}{2}}}.$$ 
	\end{definition}
	
	The following result will be used in the Section \ref{sec:unexpecdness}.
	
	\begin{proposition} 
		The sheafification of the module $D^k(\A)$ is a vector bundle  of rank $\binom{k+2}{2}$. The  sheafification of $D^k_0(\A)$ is a vector bundle of rank $k+1$, with the first Chern class equal
		$\frac{k(k+1)}{2}-|Z|.$
		
	\end{proposition}

	\begin{proof}
		
		We use the exact sequence below.  
		
		\begin{center}
			\begin{tikzpicture}
				\matrix (m) [matrix of math nodes,row sep=2.5em,column sep=3.5em,minimum width=2em,nodes={minimum height=3em, anchor=center}]
				{
					&  & 0 & 0 & \\
					& 0 & \syz((J(|Z|-1))^k) & D^k(\A) & \\
					0 & J^{k-1}fS(-1) & J^{k-1}fS(-1)\oplus S^{\binom{k+2}{2}} & S^{\binom{k+2}{2}} & 0\\
					0 & J^{k-1}fS(-1) & (J)^k(|Z|-1) & (J/f)^k(|Z|-1) & 0\\
					& 0 & 0 & 0 & \\};
				\path[-stealth]
				(m-3-1) edge node [left] {} (m-3-2)
				(m-3-2) edge node [left] {} (m-3-3)
				(m-3-3) edge node [left] {} (m-3-4)
				(m-3-4) edge node [left] {} (m-3-5)
				(m-4-1) edge node [left] {} (m-4-2)
				(m-4-2) edge node [above] {$\times f$} (m-4-3)
				(m-4-3) edge node [left] {} (m-4-4)
				(m-4-4) edge node [left] {} (m-4-5)
				(m-2-2) edge node [left] {} (m-3-2)
				(m-3-2) edge node [left] {} (m-4-2)
				(m-4-2) edge node [left] {} (m-5-2)
				(m-1-3) edge node [left] {} (m-2-3)
				(m-2-3) edge node [left] {} (m-3-3)
				(m-3-3) edge node [left] {} (m-4-3)
				(m-4-3) edge node [left] {} (m-5-3)
				(m-1-4) edge node [left] {} (m-2-4)
				(m-2-4) edge node [left] {} (m-3-4)
				(m-3-4) edge node [right] {$\phi$} (m-4-4)
				(m-4-4) edge node [left] {} (m-5-4)
				;
			\end{tikzpicture}
		\end{center}
		
		The middle column has a free module in the middle. Ther sheafification of $J^k$ is torsion-free. Thus, $\textrm{syz}(J^k)$ is after sheafification a reflexive sheaf $\widetilde{(\textrm{syz}(J^k))}$. On a surface, a reflexive sheaf is locally free (see \cite{oss}),  so
		$\widetilde{(\textrm{syz}(J^k))}$ is a vector bundle of rank $\binom{k+2}{2}$, and so is $\widetilde{D^k({\cal{A}})}$.

    We may, analogously as it was done in the case $k=1$ at the end of Appendix in \cite{CHMN}, represent $D^k(\A)$ as a direct sum, where one summand is	$SE_1\oplus \cdots \oplus SE_{\binom{k+1}{2}}$ and the other is the module of such derivations $\delta$ that $\delta(f)=0.$
Dividing $D^k(\A)$ by this first summand we get $D^k_0({\cal{A}})$.		
		
		Observe that	$\widetilde{D^k_0({\cal{A}})}$ is a vector bundle of rank $k+1$ as it arises as a division of 
		$\widetilde{D^k({\cal{A}})}$ by the sum of $\mathcal{O}_{\P^2}E_j$ for $ j=1,\dots,\binom{k+1}{2}$ and  $\bigoplus \mathcal{O}_{\P^2} E_j$  corresponds to a global  non-vanishing section.

		To get the Chern class of $\widetilde{D^k_0({\cal{A}})}$, consider the following diagram: 

		\begin{center}
			\begin{scriptsize}
				\begin{tikzpicture}
					\matrix (m) [matrix of math nodes,row sep=3em,column sep=4em,minimum width=2em, nodes={minimum height=3em, anchor=center}]
					{
						& 0 & 0 &  & \\
						0& SE_1\oplus \cdots \oplus SE_{\binom{k+1}{2}} & SE_1\oplus \cdots \oplus SE_{\binom{k+1}{2}} & 0 & \\
						0 & \D^k(\A) & S^{\binom{k+2}{2}} & (J/fS)^k(|Z|-1) & 0\\
						0 & \D^k_0(\A) & S^{\binom{k+2}{2}} \Big/ SE_1\oplus \cdots \oplus SE_{\binom{k+1}{2}} & (J/fS)^k(|Z|-1)& 0\\
						& 0 & 0 & 0 & \\};
					\path[-stealth]
					(m-2-1) edge node [left] {} (m-2-2)
					(m-2-2) edge node [left] {} (m-2-3)
					(m-2-3) edge node [left] {} (m-2-4)
					(m-3-1) edge node [left] {} (m-3-2)
					(m-3-2) edge node [left] {} (m-3-3)
					(m-3-3) edge node [above] {$\phi$} (m-3-4)
					(m-3-4) edge node [left] {} (m-3-5)
					(m-4-1) edge node [left] {} (m-4-2)
					(m-4-2) edge node [left] {} (m-4-3)
					(m-4-3) edge node [left] {} (m-4-4)
					(m-4-4) edge node [left] {} (m-4-5)
					(m-1-2) edge node [left] {} (m-2-2)
					(m-2-2) edge node [left] {} (m-3-2)
					(m-3-2) edge node [left] {} (m-4-2)
					(m-4-2) edge node [left] {} (m-5-2)
					(m-1-3) edge node [left] {} (m-2-3)
					(m-2-3) edge node [left] {} (m-3-3)
					(m-3-3) edge node [left] {} (m-4-3)
					(m-4-3) edge node [left] {} (m-5-3)
					(m-2-4) edge node [left] {} (m-3-4)
					(m-3-4) edge node [left] {} (m-4-4)
					(m-4-4) edge node [left] {} (m-5-4)
					;
				\end{tikzpicture}
			\end{scriptsize}
		\end{center}
	
We claim that
		$$c_1\big((\mathcal{J}/f\mathcal{O}_{\P^2})^k(|Z|-1) \big)=|Z|, $$ 
		where $\mathcal{J}$ is the sheafification of $J$.
		Indeed, as in \cite{CHMN}, we have the sequence
		$$0\to \mathcal{O}_{\P^2}(-|Z|)\stackrel{\cdot f}{\to} \mathcal{J}\to \mathcal{J}/f\mathcal{O}_{\P^2} \to 0.$$
		We have  $c_1(\mathcal{J}(|Z|-1))=|Z|-1$ (by the Grothendieck–Riemann–Roch theorem) so
		$c_1(\mathcal{J}(|Z|-1)/f\mathcal{O}_{\P^2})=-(-1)+(|Z|-1)=|Z|.$  
		
		Then take the $k$-th  symmetric power of the above exact sequence: 
		
		$$  0\to \mathcal{J}^{k-1}\otimes\mathcal{O}_{\P^2}(-|Z|)\stackrel{\cdot f}{\to} \mathcal{J}^k\to (\mathcal{J}/f\mathcal{O}_{\P^2})^k\to 0,  $$		(more explanations about the symmetric power and exact sequences are in Section \ref{sec:unexpecdness}).
So we get:
  $c_1(\mathcal{J}^k(|Z|-1))=|Z|-1$, $c_1(\mathcal{J}^{k-1})=0$,  $c_1(\mathcal{O}_{\P^2}(-1))=-1$, and this gives the claim.

		Next, 	we need to
		compute the first Chern class of
		$\mathcal{O}_{\P^2}^{\binom{k+2}{2}} \Big/ \mathcal{O}_{\P^2}E_1\oplus \cdots \oplus \mathcal{O}_{\P^2}E_{\binom{k+1}{2}}$.

		It is known that the sheafification of $S^3/E$ is  $T_{\P^2}(-1)$, so we prove the following lemma.
		\begin{lemma}
			$\mathcal{O}_{\P^2}^{\binom{k+2}{2}} \Big/ \mathcal{O}_{\P^2}E_1\oplus \cdots \oplus \mathcal{O}_{\P^2}E_{\binom{k+1}{2}} \cong \sym^k(T_{\P^2}(-1)).$
		\end{lemma}
	
		\begin{proof} 
			Denote by $\sim$ the permutation action of the symmetric group $S_n$. Then we have the consecutive isomorphisms 
			$$\sym^k(S^3/E) \quad \cong\quad \bigslant{\raisebox{.4em}{$\bigotimes_{i=1}^k S^3$} \Big/  \raisebox{-.4em}{$\bigoplus_{j=1}^k S^3 \otimes S^3 \otimes \cdots \otimes \overset{(j)}{E} \otimes \cdots \otimes S^3$}}{\sim} \quad \cong$$

            $$\left. \raisebox{.4em}{$\left. \raisebox{.4em}{$\bigotimes_{i=1}^k S^3$} /  \raisebox{-.4em}{$\sim$} \right.$}  \middle/ \raisebox{-.4em}{$\left. \raisebox{.4em}{$ \bigoplus_{j=1}^k S^3 \otimes S^3 \otimes \cdots \otimes \overset{(j)}{E} \otimes \cdots \otimes S^3$ } /  \raisebox{-.4em}{$\sim$} \right.$} \right. \quad \cong \quad \left. \raisebox{.4em}{}S^{\binom{k+2}{2}} \middle/ \raisebox{-.4em}{$\bigoplus_{j=1}^{\binom{k+1}{2}} SE_j$}\right .$$
			After taking the sheafification we get the assertion.
	
		\end{proof}  

		To complete the proof of the proposition we use the results on symmetric powers of a vector bundles and  their Chern classes, see eg.  \cite{Got,iena}. We obtain
		that 
		$$c_1(\sym^k(T_{\P^2}(-1))=\binom{k+1}{2}\cdot c_1(T_{\P^2}(-1))=\binom{k+1}{2}.$$
	
	\end{proof}  

\section{Unexpectedness}
	\label{sec:unexpecdness}
	In the last section, we have seen that 
	the  syzygies of $(J/f)^k(|Z|-1)$  form (after sheafification) a vector bundle of rank $k+ 1$. 
	Thus, this bundle splits as $L$ to ${\cal{O}}_L(-a_1)\oplus\dots\oplus{\cal{O}}_L(-a_{k+1})$, with $0\leq a_1\leq\dots\leq a_{k+1}$. 

	Let us remind, that the construction presented in Section \ref{sec:construction} gives a curve $C=C_L$ of degree $a_i+k$ passing through a generic point $P_L$ with multiplicity $a_i$, so this is a curve of type $(a_i+k,a_i)$.
	The next result says when  such a curve is unexpected. This result is related to Theorem 1.5 from \cite{CHMN}.
	\begin{proposition}
		Take the syzygies of $(J/f)^k(|Z|-1)$ of degree $a_i$. 
		The curve $C$ of type $(a_i+k,a_i)$ (constructed as in Section \ref{sec:construction-generalized}) is unexpected if:
		
		1. $Z$ imposes independent conditions on curves of degree $a_i+k$ and
		
		2. $(a_i+1)(k+1)\leq \sum_{j=1}^{k+1}a_{j}$. 
		
	\end{proposition}

	\begin{proof}
	
		Indeed, under our assumptions, $C$ is unexpected when
		$$\binom{a_i+k+2}{2}-|Z|-\binom{a_i+1}{2}\leq 0.$$
		
		This is equivalent to
		
		$$ka_i+a_i+\frac{k(k+3)}{2}+1\leq |Z|.$$

		Remember that ${\cal{S}}$ splits over $L$ to ${\cal{O}}_L(-a_1)\oplus\dots\oplus{\cal{O}}_L(-a_{k+1})$ with $a_1\leq\dots\leq a_{k+1}$, and this gives 
		$$a_1+\dots+a_{k+1}=-\frac{k(k+1)}{2}+|Z|.$$
	
		On the other hand, if
		$a_1+\dots+a_{k+1}=-\frac{k(k+1)}{2}+|Z|$, we have, by Assumption 2,
		$$|Z|=\frac{k(k+1)}{2}+a_1+\dots+a_{k+1}\geq \frac{k(k+1)}{2}+(k+1)(a_i+1)=ka_i+a_i+\frac{k(k+3)}{2}+1.$$
	
	\end{proof}

	The above Proposition explains, for example, the unexpectedness of the curve 
	of type $(9,7)$ for $DF_5$ (see Example \ref{ex:F3-4-5} below), or of type $(7,4)$
	for $DF_5$-without two points $(1,e,e^2), (1,e^2,1) $. However, it does not explain the unexpectedness of the curve of type $(8,5)$ for $DF_5$.
	To explain the unexpectedness of a curve of type $(d+k, d)$ with positive expected dimension, we have to prove a lemma, generalizing Lemma 3.3 from \cite{CHMN}.  
	Let us quote:
	\begin{lemma}[Lemma 3.3 of \cite{CHMN}]\label{lem:chmn3.3}
		Let $Z$ be a set of points on $\check{\P}^2$ and let $P$ be a general point on $\P^2$. Let $f$ denote, as above, the product of lines dual to the points of $Z$. Let ${\cal{S}}$ be the (rank 2) bundle of syzygies of $(J/f)(|Z|-1)$. This bundle splits on a generic line $L$ (dual to $P$), with the splitting type $(a,b)$.
		
		Then, for each integer j,
		$$\dim[I_Z+jP ]_{j+1} = \max\{0, j - a + 1\} + \max\{0, j - b + 1\}.$$
	\end{lemma}

	The generalization is  the following:
	
	\begin{lemma}\label{lem:generaliz 3.3}
		Let $Z$, $P$, $f$ and $L$ be as above. Let ${\cal{S}}$ be the (rank $k+1$) bundle of syzygies of  $(J/f)^k(|Z|-1)$.  This bundle splits on a generic line $L$ (dual to $P$), with the splitting type $(a_1,a_2,\dots a_{k+1})$.
		
		Then, for each integer j,
		$$\dim[I_Z+jP ]_{j+k} = \max\{0, j - a_1 + 1\} +\dots +\max\{0, j - a_{k+1} + 1\}.$$
		
	\end{lemma}

	\begin{proof}

		For the proof of this Lemma we need
		the construction described by Faenzi and Vall\`es in \cite{FV}. They consider the flag variety $\F =\{ (Q,l)\in \P^2\times \check{\P}^2 | Q\in l \} $. By $p, q$ they denote the projections to the first and the second factor, respectively. Then they consider the sheaf $p_*q^* I_Z(1)$ and they prove that this sheaf is isomorphic to the  logarithmic derivation bundle  $\widetilde {D_0^1(\cal{A}_Z)}$, so also it is isomorphic with the syzygies of  $(J/f)(|Z|-1)$.
		
		We want to prove an extension of this result to $k>1$, namely the following claim.
		
		\vspace{0.5cm}
		\textbf{Claim:} $$p_*q^*I_Z(k)\cong {\cal{S}}.$$
		
		Proof of the claim:
		
		{\bf Step I} The first part of the proof concerns the kernel $K$ of a map $\phi$:
		
		$$ 0\to K\to \sym^k(T_{\P^2}(-1))\stackrel{\phi}{\to} \bigoplus_{z\in Z}{\cal{O}}_{l_z}.$$
		
		We would like, similarly to 
		the argument given in \cite{FV}, prove that this kernel is unique up to an isomorphism.
		We begin with the fact already proved in \cite{FV}.	
		\begin{itemize}
			\item[1)] $\Hom(T_{\P^2}(-1),{\cal{O}}_{L_z})=\C$ 

			This follows from the fact that 
			$H^0(\Hom(T_{\P^2}(-1), {\cal{O}}_{L_z}))= H^0(T_{\P^2}(-1)^{\vee}\otimes {\cal{O}}_{L_z} )=H^0(\Omega_{\P^2}(1)\otimes {\cal{O}}_{L_z} )$.
			The cotangent sequence says:
			\begin{equation}\label{ctg}
				0\to I_{L_z}/I_{L_z}^2\to \Omega_{\P^2}\otimes {\cal{O}}_{L_z}\to \Omega_{L_z}\to 0.
			\end{equation}
			As $I_{L_z}={\cal{O}}_{L_z}(-1)$ and $ \Omega_{L_z}={\cal{O}}_{L_z}(-2)$ tensoring (\ref{ctg}) with 
			${\cal{O}}_{L_z}(1)$ we get
			\begin{equation}\label{ctg2}
				0\to {\cal{O}}_{L_z}\to \Omega_{\P^2}(1)\otimes {\cal{O}}_{L_z}\to {\cal{O}}_{L_z}(-1)\to 0.
			\end{equation}
			
			Taking the long sequence of cohomologies, we get
			$H^0(\Omega_{\P^2}(1)\otimes {\cal{O}}_{L_z} )= H^0({\cal{O}}_{\P^1})=\C$.

			\item[2)] $\Hom(\sym^k(T_{\P^2}(-1)), {\cal{O}}_{L_z})=\C$ 
			
			Here we proceed analogously as in $1)$, using the following facts:
			
			\item[a)] The dual of a symmetric power is the symmetric power of the dual space. 
			
			\item[b)] Symmetric power of a tensor product is given by the following formula:
			$ \sym^k(V\otimes W)=\bigoplus_{\lambda\vdash k} \mathbb{S}^{\lambda}V\otimes \mathbb{S}^{\lambda}W$, where $\lambda$ is a partition of $k$ giving Young tableau with at most minimum of $ \dim V,\ 
			\dim W$ rows and $\mathbb{S}$ is the Schur functor, see \cite{FH}. In our case  we will apply this formula to  $\sym^k(\Omega(1)\otimes {\cal{O}}_{L_z})$.  The only possible partition gives one row in Young tableau, and we obtain
			
			$$\sym^k(\Omega(1)\otimes {\cal{O}}_{L_z})=\sym^k(\Omega(1)) \otimes  \sym^k({\cal{O}}_{L_z})=\sym^k(\Omega(1)) \otimes  {\cal{O}}_{L_z}.$$
			
			\item[c)] Take an exact sequence of sheaves $0\to A\to B\to C\to 0.$
			Applying $\sym^k$ to this sequence, we have
			\begin{equation}\label{kernel sym}
				0\to A\otimes \sym^{k-1}B\to \sym^k B\to \sym^k C\to 0,
			\end{equation} 
			where $A\otimes \sym^{k-1}B$ means the $k$-th piece of what $A$ generates in $ \sym B$.

			Using these facts we have
			
			$\Hom(\sym^k(T_{\P^2}(-1)), {\cal{O}}_{L_z})= H^0((\sym^k(T_{\P^2}(-1))^{\vee}\otimes {\cal{O}}_{L_z} ))=H^0(\sym^k(\Omega_{\P^2}(1)\otimes {\cal{O}}_{L_z} ))$.
			Now we proceed by induction, for $k=1$	we have $H^0(\Omega_{\P^2}(1)\otimes {\cal{O}}_{L_z} )$ equal to $\C$. Assume that $H^0(\sym^j(\Omega_{\P^2}(1)\otimes {\cal{O}}_{L_z} ))=\C$ for $j<k$ take the $k$th symmetric power of the sequence
			
			\begin{equation}\label{ctg3}
				0\to {\cal{O}}_{L_z}\to \Omega_{\P^2}(1)\otimes {\cal{O}}_{L_z}\to {\cal{O}}_{L_z}(-1)\to 0,
			\end{equation}
			
			obtaining (see \ref{kernel sym})
			
			\begin{equation}\label{ctg4}
				0\to {\cal{O}}_{L_z}\otimes  \sym^{k-1}\Omega_{\P^2}(1)\to \sym^{k}(\Omega_{\P^2}(1))\otimes {\cal{O}}_{L_z}\to {\cal{O}}_{L_z}(-k)\to 0.
			\end{equation}
			
			As the global sections of ${\cal{O}}_{L_z}(-k)$ are $0$,  from the inductive assumption, we are done.
			
			\item[3)] From the above we know that all the maps from $\sym^k(T_{\P^2}(-1))$ to $\bigoplus_{z\in Z} {\cal{O}}_{L_z}$ are given by a choice of constants $(\alpha_z)_{z\in Z}$. 

			Assume now, that we chose two sets of such constants, $(\alpha_z)_{z\in Z}$ and $(\beta_z)_{z\in Z}$. Assume additionally that all the constants $\alpha_z$ and $\beta_z$ are nonzero.
			For two choices of such nonzero  constants, consider the following diagram:
			
			\[
			\begin{tikzcd}
				0 \arrow[r] & K_1 \arrow[d, " "] \arrow[r, " "] & \sym^k(T_{\P^2}(-1))\arrow[d, "="] \arrow[r, "\alpha"] & \bigoplus_{z\in Z} {\cal{O}}_{L_z} \arrow[d, "\gamma=\frac{\alpha}{\beta}"]  \\
				0 \arrow[r] & K_2\arrow[r, " "] &   \sym^k(T_{\P^2}(-1)) \arrow[r, "\beta"] & \bigoplus_{z\in Z} {\cal{O}}_{L_z} 
			\end{tikzcd}
			\]
			From this diagram (and the fact that the map $\gamma$ has an inverse, as $\alpha_z$  and $\beta_z$ are nonzero) we see that $K_1$ and $K_2$ are isomorphic.
			
			\item[4)] Take now any $z\in Z$ and the sequence
			$$0\to p_*q^*I_z(k)\to  \sym^k(T_{\P^2}(-1))\stackrel{\alpha}{\to}  {\cal{O}}_{L_z}.$$

			If $\alpha=0$, then  $p_*q^*I_z(k)\cong \sym^k(T_{\P^2}(-1))$. On the other hand, 
			from \cite[Theorem 2]{FV}, we have that $c_1(p_*q^*I_z(k))=\binom{k+1}{2}-1$,
			and we also know that $c_1(\sym^k(T_{\P^2}(-1)))=\binom{k+1}{2}$ so we get a contradiction.
			
		\end{itemize}

		To get the claim for $p_*q^*I_Z(k)$ we apply 
		$p_*q^*$  to 
		$0\to I_Z(k)\to {\cal{O}}_{\P^2}(k)\to {\cal{O}}_{Z}(k)\to 0,$
		obtaining, as in \cite{FV},
		
		$$0\to p_*q^*I_Z(k)\to \sym^k(T_{\P^2}(-1))\to \bigoplus_{z\in Z}{\cal{O}}_{l_z}.$$
		On the other hand we may also apply $\sym^k$ to the sequence:
		
		$$0\to \syz\ (J/f)(|Z|-1) \to T_{\P^2}(-1)\to (J/f)(|Z|-1)  \to 0$$
		obtaining
		
		$$0\to K \to \sym^k(T_{\P^2}(-1))\to (J/f)^k(|Z|-1) \to 0,$$
		where $K$ denotes the kernel.
		
		Thus this kernel is the bundle of syzygies of $(J/f)^k(|Z|-1) $, and $(J/f)^k(|Z|-1)  \subset \bigoplus_{z\in Z}{\cal{O}_{\ell_z}}$ (see eg. \cite{Dolgachev-Logarithmic}).

		As this kernel is unique up to isomorphism, we have
		$$K\cong p_*q^*I_Z(k),$$
		what proves the claim.
		
		To get a more specific description of this kernel,  we use 
 formula (\ref{kernel sym}), obtaining
		
		$$\widetilde{\syz}( (J/f)^{k}(|Z|-1))=\widetilde{\syz}\ (J/f)(|Z|-1)\otimes \sym^{k-1}(T_{\P^2}(-1))$$ 
		and so
		$$p_*q^*I_Z(k)=p_*q^*I_Z(1)\otimes \sym^{k-1}(T_{\P^2}(-1)).$$

		Having the claim   we proceed exactly as it is done in Lemma 3.3 of \cite{CHMN}.
		
		Let us, for the reader's convenience, go through this part of the  proof. Observe, as it is done  in \cite{CHMN}, that  $q$ restricted to the set $Y=\{(Q,\ell): Q\in L=L_P\}\subset \mathbb{F}$, where $P$ is the generic point, may be treated as a blowup of $\check{\P}^2$ in $P$. 
		
		So $q^*(I_Z(k))$ may be treated as a sheaf on $Y$ given 
		by $I_Z\otimes \O_Y(kH)$, with $H$ being a pullback of a line, so
		$ p_*(I_Z\otimes \O_Y((j+k)H-jE)=p_*(I_Z\otimes \O_Y((kH)\otimes p^*(\O_{L}(j))\cong p_*(I_Z\otimes \O_Y((kH)\otimes \O_L(j)))$ (by the projection 	formula).	
		
		On the other hand the projection $p$ maps $Y$ onto $L$ as $\P^1$ bundle. According to Theorem 2 from \cite{FV},  $p_*q^*I_Z(k)$ is a vector bundle of rank $k+1$ and as such decomposes, after restriction to $L$, as a sum of line bundles, say $\sum_{i=1}^{k+1}\O_L(-a_i)$.
		
		Thus we get that $p_*q^*I_Z(j+k)$ restricted to $L$ is
		$$\sum_{i=1}^{k+1}\O_L(j-a_i).$$
		
		Computing the appropriate dimensions  we have:
		$$ \dim[I_{Z+jP}]_{j+k}=h^0(\P^2, I_{Z+jP}\otimes {\cal{O}}_{\P^2}(j+k))=$$
		$$=h^0(\P^2, I_{Z}\otimes I_{jP}\otimes {\cal{O}}_{\P^2}(j+k))
		=h^0(Y,I_Z\otimes \O_Y((j+k)H-jE))
		$$	 
		This, from Leray spectral sequence (and using the fact that $L$ is generic, so $R^1p_*(I_Z\otimes \ell_Q), Q\in L$ vanishes, according to Theorem 2 from \cite{FV})	 
		equals
		$$ 
		h^0(L,p_*(I_Z\otimes \O_Y((j+k)H-jE))=h^0(L,p_*(I_Z(k)\otimes \O_Y((jH-jE))=$$
		
		$$=h^0(L,p_*(I_Z(k)\otimes p^*\O_L(j))),$$
		
		and from the projection formula it is
		$$=h^0(L,p_*(I_Z(k))\otimes \O_L(j) )=
		h^0(L,\oplus_i\O_L(j-a_i)).
		$$	
		
	\end{proof}
	
	Now we are in the position to prove a result describing when a curve $C_L$, constructed as in Section \ref{sec:construction} is unexpected.
	\begin{proposition}\label{pro:unexp}
		
		Let $Z$, $P$, $f$ be as above. Take  the (rank $k+1$) bundle of syzygies of $(J/f)^k(|Z|-1)$. This bundle splits on a generic line $L$ (dual to $P$), with the splitting type $(a_1,a_2,\dots a_{k+1})$.  Let us introduce the following notation:
		$$  (a_1,a_2,\dots a_{k+1})=
		(a,\dots,a,a+\epsilon_1,\dots,a+\epsilon_1,a+\epsilon_2,\dots,a+\epsilon_2,a+\epsilon_3,\dots,a+\epsilon_s)
		$$
		where $\epsilon_0=0, 1\leq\epsilon_1$ and 
		$\epsilon_i<\epsilon_{i+1}$, and $a+\epsilon_i$, for $i=0,1,\dots,s$, appears in the sequence $t_i$ times,  $t_0+\dots+t_s=k+1$. 
		Take syzygies of $(J/f)^k(|Z|-1) $, of degree $a+\epsilon_j$, for a given $j\in\{0,1,\dots, s\}$.
		The curve $C_L$ of type $(a+\epsilon_j+k, a+\epsilon_j)$  is unexpected if:

		1. $Z$ imposes independent conditions on curves of degree $a+\epsilon_j+k$ and
		
		2. $0<\sum_{i=j+1}^st_i(\epsilon_i-\epsilon_j-1).$
	\end{proposition}
	
	\begin{proof}
		From Lemma \ref{lem:generaliz 3.3} it follows, that
		$$\dim[I_Z+(a+\epsilon_j)P ]_{a+\epsilon_j+k} = \max\{0, a+\epsilon_j - a_1 + 1\} +\dots+ \max\{0, a+\epsilon_j - a_{k+1} + 1\}=$$
		$$=(\epsilon_j+1)t_0+(\epsilon_j+1-\epsilon_1)t_1+\dots+(\epsilon_j+1-\epsilon_j)t_j .$$
		On the other hand, the expected dimension is 
		$$\binom{a+\epsilon_j+k+2}{2}-|Z|-\binom{a+\epsilon_j+1}{2}.$$
		We also know that $a_1+\cdots+a_{k+1}=|Z|-\binom{k+1}{2}.$
		Thus, the expected dimension is less than the real one iff
		$$\binom{a+\epsilon_j+k+2}{2}-(\binom{k+1}{2}+(k+1)a+\sum_{i=1}^st_i\epsilon_i)-\binom{a+\epsilon_j+1}{2}<$$
		
		$$< (t_0+\dots+t_j)(\epsilon_j+1)-\sum_{i=1}^st_i\epsilon_i,$$
		what  is equivalent to
		
		$$(k+1)(\epsilon_j+1)\sum_{i=1}^st_i\epsilon_i< (t_0+\dots+t_j)(\epsilon_j+1)-\sum_{i=1}^st_i\epsilon_i.$$
		So, as $t_0+\dots+t_s=k+1$ and $\epsilon_0=0$ we have equivalently
		$$\sum_{i=j+1}^st_i(\epsilon_j+1)< \sum_{i=j+1}^st_i\epsilon_i,$$
		and thus
		$$0<\sum_{i=j+1}^st_i(\epsilon_i-\epsilon_j-1).$$
	
	\end{proof}
	
	\begin{remark}\label{re:unexp}
		It may, perhaps, happen that the  dimension of a system of curves of type $(d+k,d)$ passing once through $Z$ is equal to the expected dimension, but there is an unexpected curve of this type, with multiplicity greater than one in some points of $Z$. 
		
		In Example \ref{ex:F3-4-5} there are three linearly independent curves of type $(7,5)$ for $DF_4$ arrangement. As far as Singular \cite{Singular} can check, they are irreducible.  Moreover, one of them pass doubly through two points of $Z$, so the expected dimension count should take this under consideration.  
		
	\end{remark}

\section{Examples}
	
	This section presents some examples which were the starting point for the considerations.
	
	\begin{example}\label{ex:B3k=1}
		Here we show how the  construction 
		of the unexpected curve works in case of $B_3$ configuration and
		for $k=1$. 
		
		Take the syzygies of the jacobian of $f=abc(a^2-b^2)(a^2-c^2)(b^2-c^2)$
		given by
		$$g_0(a,b,c)= 4 a^3 - 5 a b^2 - 5 a c^2$$
		$$g_1(a,b,c) = -5 a^2 b + 4 b^3 - 5 b c^2$$
		$$g_2(a,b,c) = -5 a^2 c - 5 b^2 c + 4 c^3 $$
		so that $g_0(a,b,c)f_a(a,b,c)+g_1(a,b,c)f_b(a,b,c)+g_2(a,b,c)f_c(a,b,c)=0$.
		
		Take a generic line $L$ in the plane
		
		$$\al a+\be b+\ga c=0.$$
		
		Take then the line 
		$$L_G: g_0(a,b,c)x+g_1(a,b,c)y+g_2(a,b,c)z=0$$ 
		in the dual projective plane, and, for any  point $Q=(a,b,c)\in L$ the dual line
		$$L_Q: ax +by+cz=0.$$
		
		Assume that $c=1$.
		Compute then $$b=\frac{-\al a-\ga}{\be},$$
		substitute into the equation of $L_G$ and multiply by $\be^3$. We get:
		$$x(-5 a^3 \al^2 \be - 5 a \be^3 + 4 a^3 \be^3 - 10 a^2 \al\be\ga - 5 a\be\ga^2)+$$
		$$y(-4 a^3 \al^3 + 5 a \al \be^2 + 5 a^3 \al \be^2 - 12 a^2 \al^2 S + 5\be^2 \ga + 
		5 a^2 \be^2 \ga - 12 a \al \ga^2 - 4 \ga^3)+$$
		$$z(-4 a^3 \al^3 + 5 a \al\be^2 + 5 a^3 \al\be^2 - 12 a^2 \al^2 \ga + 5 \be^2 \ga + 
		5 a^2 \be^2 \ga - 12 a \al\ga^2 - 4 \ga^3)=0, $$
		or
		\begin{eqnarray}\label{eqn:aB3k=1}
		    \begin{aligned}
			a^3(-5 \al^2 \be x + 4 \be^3 x - 4 \al^3 y + 5 \al \be^2 y)+a^2(-10 \al\be\ga x - 12 \al^2 \ga y + 5 \be^2 \ga y - 5 \al^2 \be z - 5 \be^3 z)+\\
			+a(-5 \be^3 x - 5 \be \ga^2 x + 5 \al \be^2 y - 12 \al\ga^2 y - 10 \al\be\ga z)+5\be^2 \ga y - 4 \ga^3 y + 4 \be^3 z - 5\be\ga^2 z=0.
	    	\end{aligned}
		\end{eqnarray}
		
		Then, for any point $(x,y,z)$ different from $(\al,\be,\ga)$, we get from the equations of $L$ and $L_Q$:
		$$a=\frac{\ga y-\be z}{\be x-\al y}.$$
		
		Substituting this for $a$ in (\ref{eqn:aB3k=1})
		and multiplying by the denominator in the third power we get:
		
		\begin{eqnarray*}
			9\be^3(-\ga^3 x^3 y + \ga^3 x y^3 + \be^3 x^3 z - 3 \al \be^2 x^2 y z + 
			3 \al \ga^2 x^2 y z + 3 \al^2 \be x y^2 z\\ - 3 \be \ga^2 x y^2 z - \al^3 y^3 z - 
			3 \al^2  x y z^2 +3  \be^2 \ga x y z^2 - \be^3 x z^3 + \al^3 y z^3)=0
		\end{eqnarray*}
		
		The expression in parenthesis is the (equation of the) unexpected quartic 
		with a generic triple point described in \cite{CHMN} and in \cite{BMSS}.
	\end{example}

	\begin{example}
		\label{ex:B3k=2}
		The theory developed in Section \ref{sec:derivations} allows us to compute the module   $\syz(J^2+L)$, in case of $B_3$ configuration, which is generated by three elements
		$$[[0,0,0,0,y^2-z^2,0], \; [0,y^2,0,xy,xz,0], \;[0,0,z^2,0,xy,xz]]=(\sigma_1,\sigma_2,\sigma_3).$$
		Thus
		$$D^2_0(B_3) \otimes \O_L=\O_L(-2)\otimes \O_L(-2) \otimes \O_L(-2).$$
		If we take as a general line, the line with equation $L=-12x+10y+7z$ and syzygy $\sigma_2$, then Theorem \ref{thm:main} gives the equation of curve to be
		$$49x^3y-49xy^3+168x^2yz+140xy^2z+44xyz^2=0.$$
	\end{example}
	
	\begin{example}
		\label{ex:F3-4-5}
		Let $e$ be the $n$-th primitive root from unity. Denote by $DF_n=xyz\prod_{i,j=0}^{n-1}(x+e^iy+e^jz)$  configuration of lines dual to the points cut by the so-called Fermat configuration of lines $(x^n-y^n)(x^n-z^n)(y^n-z^n)$. Fermat configurations contains exactly $n^2+3$ points and we set $Z$ to be the set of those points. The following tables give the exponents $a_i$ in the splitting type, i.e.
		$$D^k_0(DF_n) \otimes \O_L=\O_L(-a_{1})\otimes \cdots \otimes \O_L(-a_{k+1}),$$
		for $n=3,4,5$ and all $k$ which fulfil inequality $n^2+3>\binom{k+1}{2}$, together with the values of $\epsilon_i$ and $t_i$ described in Proposition \ref{pro:unexp}. The last column contains all unexpected curves of type $(d+k,d)$ which can be obtained by this proposition. For the readers convenience we adopt here the convention that we give an exact number of values of $\epsilon_i$ and $t_i$. Therefore, if some values do not exist, we put $0$ instead of omitting.
		
		\vspace{0.5cm}
		\begin{minipage}{0.97\textwidth}
			\centering
			$n=3$	\begin{tabular}[H]{c|c|c|c|c}
				$k$ & $a_1,\ldots,a_{k+1} $ & $\epsilon_1$ & $t_0,t_1$ & $(d+k,d)$\\
				\hline
				\rule{0pt}{2ex} 
				1& 4,7 & 3 & 1,1 & (5,4)\\
				\rule{0pt}{2ex} 
				2& 3,3,3 & 0 & 0,0 & ---\\
				\rule{0pt}{2ex} 
				3& 1,1,2,2 & 1 & 2,2 & ---\\
				\rule{0pt}{2ex} 
				4& 0,0,0,1,1 & 1 & 3,2 & ---\\
				\rule{0pt}{2ex} 
				5& 0,0,0,0,0,1 & 1 & 5,1 & ---\\
			\end{tabular}
		\end{minipage}
		
		\vspace{0.5cm}
		\begin{minipage}{0.97\textwidth}
			\centering
			$n=4$		\begin{tabular}[H]{c|c|c|c|c}
				$k$ & $a_1,\ldots,a_{k+1} $ & $\epsilon_1,\epsilon_2$ & $t_0,t_1,t_2$ & $(d+k,d)$\\
				\hline
				\rule{0pt}{2ex} 
				1& 9,9 & --- & --- & ---\\
				\rule{0pt}{2ex} 
				2& 4,5,7 & 1,3 & 1,1,1 & (6,4),(7,5)\\
				\rule{0pt}{2ex} 
				3& 3,3,3,4 & 1,0 & 3,1,0 & ---\\
				\rule{0pt}{2ex} 
				4& 1,1,2,2,3 & 1,2 & 2,2,1 & ---\\
				\rule{0pt}{2ex} 
				5 & 0,0,0,1,1,2 & 1,2 & 3,2,1 & ---\\
				\rule{0pt}{2ex}
				6 & 0,0,0,0,0,0,1 & 1,0 & 6,1,0 & ---\\
			\end{tabular}
		\end{minipage}
		
		\vspace{0.5cm}
		\begin{minipage}{0.97\textwidth}
			\centering
			$n=5$		\begin{tabular}[H]{c|c|c|c|c}
				$k$ & $a_1,\ldots,a_{k+1} $ & $\epsilon_1,\epsilon_2,\epsilon_3$ & $t_0,t_1,t_2,t_3$ & $(d+k,d)$\\
				\hline
				\rule{0pt}{2ex} 
				1& 13,14 & 1,0,0 & 1,1,0,0 & ---\\
				\rule{0pt}{2ex} 
				2& 7,9,9 & 2,0,0 & 1,2,0,0 & (9,7)\\
				\rule{0pt}{2ex} 
				3& 4,5,6,7 & 1,2,3 & 1,1,1,1 & (7,4)$^*$,(8,5)\\
				\rule{0pt}{2ex} 
				4& 3,3,3,4,5 & 1,2,0 & 3,1,1,0 & (7,3)$^*$\\
				\rule{0pt}{2ex} 
				5 & 1,1,2,2,3,4 & 1,2,3 & 2,2,1,1 & (6,1)$^*$, (7,2)$^*$\\
				\rule{0pt}{2ex}
				6 & 0,0,0,1,1,2,3 & 1,2,3 & 3,2,1,1 & ---\\
				\rule{0pt}{2ex}
				7 & 0,0,0,0,0,0,0,2 & 2,0,0 & 7,1,0,0 & ---\\
				\rule{0pt}{2ex}
				8 & 0,0,0,0,0,0,0,0,1 & 1,0,0 & 8,1,0,0 & ---\\
			\end{tabular}
			\\
			\vspace{2mm}
			{\tiny		
				$^*$ means a case when the conditions imposed by $Z$ are dependent.}		
		\end{minipage}
		
		\vspace{0.5cm}
		Some interesting examples can be found among all given cases. Consider for instance the curve $(7,5)$ for $n=4$. As is computed, the curve constructed by Theorem \ref{thm:main} for this case has an unusual property. Namely, the curve passes through
		all points of the set $Z$ once, except points $(0,1,0)$ and $(0,0,1)$, which are double. The equation of this curve, where the general point has coordinates $(a,b,c)$, is
		
		\begin{scriptsize}
			\begin{multline*}
				\mathcal{C}_{4,7,5}=(5b^4c+3c^5)x^6y+(-20ab^3c)x^5y^2+(30a^2b^2c)x^4y^3+(-20a^3bc)x^3y^4+(5a^4c-3c^5)x^2y^5+(-3b^5-5bc^4)x^6z+\\
				(10ab^4-10ac^4)x^5yz+(-10a^2b^3)x^4y^2z+(5a^4b+5bc^4)x^2y^4z+(-2a^5+10ac^4)xy^5z+(20abc^3)x^5z^2+(10a^2c^3)x^4yz^2+(-20abc^3)xy^4z^2+\\
				(-10a^2c^3)y^5z^2+(-30a^2bc^2)x^4z^3+(30a^2bc^2)y^4z^3+(20a^3bc)x^3z^4+(-5a^4c-5b^4c)x^2yz^4+(20ab^3c)xy^2z^4+(-30a^2b^2c)y^3z^4\\
				+(-5a^4b+3b^5)x^2z^5+(2a^5-10ab^4)xyz^5+(10a^2b^3)y^2z^5=0.
			\end{multline*}
		\end{scriptsize}
		On the other hand, if we consider the system $\mathcal{L}$ of all curves which pass once through all points dual to $DF_4$ and which pass through a general point with multiplicity $5$, then we can compute that $\dim [\mathcal{L}]_7=3$, while the expected dimension is
		$$\binom{9}{2} -|Z| - \binom{6}{2} = 36-19-15=2.$$
		Therefore, there exists an unexpected curve of type $(7,5)$ different from what we got previously from Theorem \ref{thm:main}. By using computer algebra software it can be computed that the equation of such a curve is
		
		\begin{scriptsize}
			\begin{multline*}
				\mathcal{C'}_{4,7,5}=(50ab^6c^2+90ab^2c^6)x^4y^3+(-150a^2b^5c^2-90a^2bc^6)x^3y^4+(150a^3b^4c^2+30a^3c^6)x^2y^5+(-50a^4b^3c^2-30b^3c^6)xy^6+\\
				(-30b^8c-60b^4c^5-6c^9)x^5yz+(50ab^7c-110ab^3c^5)x^4y^2z+(60a^3bc^5)x^2y^4z+(-50a^4b^4c-30a^4c^5+90b^4c^5+6c^9)xy^5z+\\
				(30a^5b^3c+50ab^3c^5)y^6z+(15b^9+66b^5c^4+15bc^8)x^5z^2+(-25ab^8+190ab^4c^4+15ac^8)x^4yz^2+(25a^4b^5-15a^4bc^4-75b^5c^4-15bc^8)xy^4z^2+\\
				(-15a^5b^4+9a^5c^4-175ab^4c^4-15ac^8)y^5z^2+(-200ab^5c^3-60abc^7)x^4z^3+(200ab^5c^3+60abc^7)y^4z^3+(150a^2b^5c^2+90a^2bc^6)x^3z^4+\\
				(-150a^3b^4c^2-30a^3c^6)x^2yz^4+(50a^4b^3c^2+30b^3c^6)xy^2z^4+(-50ab^6c^2-90ab^2c^6)y^3z^4+(-60a^3bc^5)x^2z^5+\\
				(50a^4b^4c+30a^4c^5+30b^8c-30b^4c^5)xyz^5+(-30a^5b^3c-50ab^7c+60ab^3c^5)y^2z^5+
				(-25a^4b^5+15a^4bc^4-15b^9+9b^5c^4)xz^6+\\
				(15a^5b^4-9a^5c^4+25ab^8-15ab^4c^4)yz^6=0,
			\end{multline*}
		\end{scriptsize}
		where the general point has coordinates $(a,b,c)$.
		
		This example suggests that perhaps not all unexpected curves of type $(d+k,d)$ can be derived from syzygies or different construction should be used.
		
		A similar situation takes place for $n=5$ and the curve of type $(8,5)$. The construction of Theorem \ref{thm:main} leads to a curve with $2$ double points among the set of points dual to $DF_5$. The explicit equation of this curve with general point $(a,b,c)$ is
		
		\begin{scriptsize}
			\begin{multline*}
				\mathcal{C}_{5,8,5}=(3b^5c+2c^6)x^7y+(-10ab^4c)x^6y^2+(10a^2b^3c)x^5y^3+(-5a^4bc)x^3y^5+(2a^5c-2c^6)x^2y^6+(-2b^6-3bc^5)x^7z+(6ab^5-6ac^5)x^6yz+\\
				(-5a^2b^4)x^5y^2z+(2a^5b+3bc^5)x^2y^5z+(-a^6+6ac^5)xy^6z+(10abc^4)x^6z^2+(5a^2c^4)x^5yz^2+(-10abc^4)xy^5z^2+(-5a^2c^4)y^6z^2+\\
				(-10a^2bc^3)x^5z^3+(10a^2bc^3)y^5z^3+(5a^4bc)x^3z^5+(-2a^5c-3b^5c)x^2yz^5+(10ab^4c)xy^2z^5+(-10a^2b^3c)y^3z^5+(-2a^5b+2b^6)x^2z^6+\\
				(a^6-6ab^5)xyz^6+(5a^2b^4)y^2z^6=0,
			\end{multline*}
		\end{scriptsize}
		whereas we can again find a curve passing simply through $Z$ and through general point with multiplicity $5$, and such a curve has a different equation.
		
	\end{example}
	\subsection*{Acknowledgements.}
	We would like to warmly thank Jakub Byszewski, Marcin Dumnicki, Brian Harbourne, Piotr Pokora, Tomasz Szemberg and Jerzy Weyman for help and discussions.
	
	The first author was partially supported by National Science Center (Poland) Sonata Grant Nr
	2018/31/D/ST1/00177. The second author was partially supported by 
	National Science Center grant Opus 2019/35/B/ST1/00723.

	\footnotesize
	\noindent
	Grzegorz Malara: Department of Mathematics, Pedagogical University of Cracow, 	Podchor\c a\.zych 2, 30-084 Krak\'ow, Poland, grzegorz.malara$@$up.krakow.pl, grzegorzmalara@gmail.com 
	\\
	Halszka Tutaj-Gasi\'nska: Faculty of Mathematics and Computer Science, Jagiellonian University, Stanis{\l}awa {\L}ojasiewicza 6, 30-348 Kraków, Poland,
	halszka.tutajgasinska@gmail.com
	
\end{document}